\documentclass[a4paper,10pt]{article}
\usepackage[utf8]{inputenc}
\usepackage[T1]{fontenc}
\usepackage{lmodern}
\usepackage{textcmds}
\usepackage{epsfig}
\usepackage{amsfonts}
\usepackage{amsthm}
\usepackage{latexsym}
\usepackage{amsmath}
\usepackage{amscd}
\usepackage{amssymb}
\usepackage{enumerate}
\usepackage{textcomp}

\theoremstyle{plain}
\newtheorem{lm}{Lemma}[section]
\newtheorem{thm}[lm]{Theorem}
\newtheorem{cor}[lm]{Corollary}
\newtheorem{lem}[lm]{Lemma}
\newtheorem{prop}[lm]{Proposition}

\theoremstyle{definition}

\newtheorem{defi}[lm]{Definition}

\theoremstyle{remark}
\newtheorem{rem}[lm]{Remark}
\newtheorem{note}[lm]{Notation}
\newtheorem{case}{Case}

\DeclareMathOperator{\A}{\mathrm{Aut}}
\DeclareMathOperator{\G}{\mathrm{GL}}
\DeclareMathOperator{\PG}{\mathrm{PGL}}

\title{A note on the structure theory of groups acting on the Bruhat Tits tree associated to $\PG_{2}(F)$}
\author{Inder Kaur\footnote {This note is part of the author's master's thesis for Humboldt Universit\"{a}t zu Berlin, submitted on 23rd of February, 2012}}
\date {}
\begin{document}

\maketitle

\begin{abstract}
In this note we prove the classical decomposition theorems for certain subgroups of the automorphism group of the Bruhat-Tits tree associated to $\PG_{2}(F)$, where $F$ is a local field. 
The result can be applied to \textit{any} closed transitive subgroup of the automorphism group of \textit{any} Bruhat-Tits tree. \vspace{5mm} \end{abstract}

\section{Introduction}
We will use the following notation throughout this article.
\begin{note}
Let $p \geq 2$ be a prime number, $F$ a local field with valuation $\nu$ and uniformizer $\varpi$. Denote by $\mathfrak{o}$ the discrete valuation ring, $\mathfrak{p}$ its maximal ideal, $U$ the group of units and by $\mathbf{k}$ the finite residue field with characteristic $p$.
Let $\G_{2}(F)$ be the group of $2 \times 2$ invertible matrices with coefficients in $F$. The finite field with $p$ elements is denoted by $\mathbb{F}_{p}$ and its algebraic closure is denoted $\overline{\mathbb{F}}_{p}$.
We denote by $X$, an arbitrary Bruhat-Tits tree and by $X_{\scriptscriptstyle{\PG_{2}(F)}}$, the Bruhat-Tits tree associated to $\PG_{2}(F)$. 
\end{note}

We equip $X$ with its natural metric (see section $3$) and denote by $\A(X)$, the isometry group of $X$.
There is a natural action of  $\G_{2}(F)$  on $X_{\scriptscriptstyle{\PG_{2}(F)}}$ and more precisely, this action induces an embedding of $\PG_{2}(F)$ into the group of isometries of the tree $X_{\scriptscriptstyle{\PG_{2}(F)}}$, denoted by $\A(X_{\scriptscriptstyle{\PG_{2}(F)}})$ (see section $3$).
The purpose of this note is to study the decomposition theorems for certain subgroups of the group $\A(X_{\scriptscriptstyle{\PG_{2}(F)}})$.

\vspace{0.2 cm}
The motivation for this work comes from representation theory. 
In $1976$, G. Ol'shanskii \cite{O} classified all irreducible smooth representations of $\A(X)$ on $\mathbb{C}$-vector spaces.  
Let us consider the tree  $X_{\scriptscriptstyle{\PG_{2}(F)}}$. We can ask the following natural question: how can one classify the irreducible smooth representations of $\A(X_{\scriptscriptstyle{\PG_{2}(F)}})$ over $\overline{\mathbb{F}}_{p}$-vector spaces?
To answer this, one idea is to classify the representations of certain subgroups $\mbox{\^{G}}^{(e)}$ (defined in section $4$) of $\A(X_{\scriptscriptstyle{\PG_{2}(F)}})$ (depending on a `ramification' parameter $e \in \mathbb{N}$) which contain $\PG_{2}(F)$.  
One hopes that the representation theory of these groups is similar to that of $\PG_{2}(F)$ as studied in \cite {B-L1} and \cite{B-L2}.  
In order for an analogy to hold, the first step is to look for similarities in the structure theory of the groups $\mbox{\^{G}}^{(e)}$ and $\PG_{2}(F)$. 
This is precisely what is covered in our main result (see Theorem \ref{main thm}); we prove the decomposition theorems of $\mbox{\^{G}}^{(e)}$ analogous to that of $\PG_{2}(F)$.

\vspace{0.2 cm}
\textit{Outline}: In section $2$ we recall the classical decompositions of $\PG_{2}(F)$ and the Bruhat-Tits tree associated to it. 
In section $3$, we discuss some basic properties of the group $\A(X)$. In particular, we define the property of weak two transitivity and formulate an equivalent criterion for any closed, transitive subgroup of $\A(X)$ to be weakly two-transitive (Proposition \ref{weak 2 transitivity}).
In section $4$, we define $\mbox{\^{G}}^{(e)}$ and prove that it is a closed subgroup of $\A(X_{\scriptscriptstyle{\PG_{2}(F)}})$. 
We show that it is weakly two-transitive using Proposition \ref{weak 2 transitivity} and later use this property to prove its classical decompositions (Theorem \ref{main thm}).

\vspace{0.2 cm}
\textit{Acknowledgements}: 
The results in this note are part of my master's thesis which was supervised by Prof. Dr. Elmar Grosse-Kl\"{o}nne. I thank him for providing the question. I am grateful to Prof. Dr. Ernst Wilhelm-Zink for his careful reading of the thesis and his positive feedback and encouragement. 
I thank the IMU Berlin Einstein Foundation Program for funding during the writing period and the Berlin Mathematical School for providing excellent working facilities.

\section{Basic Definitions}
In this section we briefly recall some basic definitions and results that will be used in the rest of the article.

\subsection{The Group $\PG_{2}(F)$}

We begin by discussing the decompositions of $\G_{2}(F)$. 

\begin{defi}\label{subgroups}
Let $G:= \mathrm{GL}_2(F)$ and $K:= \mathrm{GL}_2(\mathfrak{o})$.
The group $G$ has the following important subgroups:

\begin{enumerate}
\item The standard Borel subgroup $B := \left\{\begin{array}{lr}\left(\begin{array}{ccc} a & b \\0 & d \end{array}\right)\in G\end{array}\right\}$.
\item The unipotent radical $N$ of $B$,  $N := \left\{\begin{array}{lr}\left(\begin{array}{ccc} 1 & b \\ 0 & 1 \end{array}\right)\in G\end{array}\right\}$.
\item The subgroup $N^{\prime} := \left\{\begin{array}{lr}\left(\begin{array}{ccc} 1 & 0 \\ b & 1 \end{array}\right)\in G\end{array}\right\}$. 
\item The centre $Z$ of $G$,  $Z := \left\{\begin{array}{lr}\left(\begin{array}{ccc} a & 0 \\0 & a \end{array}\right)\in G\end{array}\right\}$.
\item The standard split maximal torus $T := \left\{\begin{array}{lr}\left(\begin{array}{ccc} a & 0 \\0 & b \end{array}\right)\in G\end{array}\right\}$.
\item The Iwahori subgroup $I := \left\{\begin{array}{lr}\left(\begin{array}{ccc} a & b \\c & d \end{array}\right) \in K : a, d \in U,  \ c \in \mathfrak{p}, \ b \in \mathfrak{o} \end{array}\right\}$.

\end{enumerate} 
\end{defi}

\begin{rem}\label{K cpt}
The group $K$ is the unique maximal compact subgroup of $G$ up to conjugation. We refer the reader to \cite[pp $304$]{B-K}, for a detailed proof of this fact.  
\end{rem}

The following theorem states the classical decompositions of $G$.

\begin{thm}{\label{thm 1}} Denote by $s$ the permutation matrix $\left(\begin{array}{ccc} 0 & 1 \\ 1 & 0 \end{array}\right)$.
Using the notations defined in \ref{subgroups}, we have the following:

\begin{enumerate}

\item The Iwasawa decomposition: $G = BK$.

\item Let $\varpi$ be a prime element of $F$. The matrices
 
\[\left(\begin{array}{ccc} \varpi^{a} & 0 \\ 0 & \varpi^{b}\end{array}\right),\ \forall a, b \in \mathbb{Z},  \ \ \ a<b \]

form a set of representatives for the coset space $K \backslash G/K$ and $G$ has the Cartan decomposition 
\[ G = \displaystyle {\bigsqcup_{a\leq b}} K \left(\begin{array}{ccc} \varpi^{a} & 0 \\ 0 & \varpi^{b}\end{array}\right) K. \]

\item The Bruhat decomposition  $G = B \cup BsB$.

\item The Levi decomposition $B = NT$.

\item There is a decomposition for the Iwahori subgroup $I$,  $I = (I \cap N^{\prime})(I \cap T)(I \cap N)$.
More precisely, the product map \[I \cap N^{\prime}\times I \cap T \times I \cap N \longrightarrow I \] is a homeomorphism, for any ordering of the factors on the left hand side.

\end{enumerate}
 
\end{thm}
\begin{proof}
By matrix manipulation. See \cite[pp $50-52$]{B-H} for full details.
\end{proof}

\subsection{The Bruhat-Tits tree}

We now recall the definition of a Bruhat-Tits tree and some of its basic properties. We refer the reader to \cite{S} and \cite{Ch} for exhaustive treatments of the subject. 

\begin{defi}
A \textit{tree} $X = (X^{0},X^{1})$ (where $X^{0}$ denotes the set of vertices and $X^{1}$ the set of edges) is a non-empty connected graph without circuits/cycles.
A tree is called \textit{Bruhat-Tits} if every vertex belongs to $m+1$ distinct edges for a fixed $m \in \mathbb{N}$ and $m \geq 2$.
\end{defi}

\begin{defi} 
A \textit{path} in $X$ is a finite or infinite sequence of distinct vertices $x_{i}$ connected by a sequence of edges $\{x_{i},  x_{i+1}\}$.
\end{defi}

This leads us to the following definition of distance on $X$.

\begin{defi}
Given any vertices $x,y \in X^{0}$, there exists a unique path joining $x$ to $y$, since $X$ is connected. We define the \textit{distance} ${d(x,y)}$ to be the number of edges in this path.
For every vertex $x \in X^{0}$, we set $d(x, x) = 0$.
\end{defi}

\begin{defi}
A path in a tree is called \textit{infinite} if it is parametrised by $\mathbb{N}$ and \textit{doubly infinite} if it is parametrised by $\mathbb{Z}$. 
\end{defi}

\begin{defi} An \textit{end} of the tree $X$ is an equivalence class of infinite paths for the equivalence relation $\sim$ given by  
 \[ (x_{i})\sim(y_{j}) \Longleftrightarrow ( \exists \ k,  \ p : x_{i} = y_{i+k}, \  \forall \ i > p). \] 
\end{defi}

In other words, two infinite paths are equivalent (and therefore belong to the same end) if they differ only by a finite initial sequence of vertices. 
We can think of an infinite path as an infinite ray in the tree heading off to infinity.
Then an `end' represents a point at infinity for the tree $X$. 

\begin{note} 
We denote by $\Omega$ the set of ends of $X$. 
We use the notation $(x_{i})_{i \in \mathbb{N}}$ for any infinite path and $|(x_{i})|_{i \in \mathbb{N}}$ for an end $\omega$.
The notation $|\mbox{\textperiodcentered}|$ here signifies that $\omega$ is an equivalence class of infinite paths.
\end{note}

\begin{defi}
A doubly infinite path $[\omega^{\prime},  \omega]$ joining two ends $\omega^{\prime},\omega$ is called an \textit{apartment}.
\end{defi}
 
\begin{rem} \label{apart} The following are easy to check.
\begin{enumerate}
\item Given $x \in X^{0}, \omega \in \Omega$, there exists a unique infinite path $[x, \omega]$, starting from $x$ which belongs to the equivalence class $\omega$. 
\item Given any two distinct ends $\omega, \omega^{\prime}$, there exists an apartment $[\omega,\omega^{\prime}]$ containing them.
\end{enumerate}
\end{rem}

We now state without proof a result we will be using in the proof of the main theorem. 

\begin{prop}[{\cite[pp $19$]{Ch}}] \label{crossroad 2} 
Any three distinct ends of $X$ meet at a unique vertex called the \textit{crossroad}.  
\end{prop}

\subsection{The Bruhat-Tits tree associated to $\PG_{2}(F)$}

In this subsection we construct the Bruhat-Tits tree $X_{\PG_{2}(F)}$, associated to $\PG_{2}(F)$ and analyse the action of $\PG_{2}(F)$ on it.   
The reader is referred to \cite{S} and \cite{DT} for complete details.

\vspace{0.2 cm}
\noindent Denote by $V$ an $F$-vector space of dimension 2.

\begin{defi} A lattice $L$ of $V$ is a finitely generated $\mathfrak{o}$-submodule of $V$ which generates the $F$-vector space $V$. 
It is a free module of rank 2.\end{defi}

\begin{defi}\label{vert}
The group $F^\times$ acts on the set of lattices of $V$ as follows. 
For $m \in (F^{\times},\times)$ and $L$ a lattice of $V$,  $mL$ is also a lattice of $V$. 
The orbit of a lattice under this action is called its \textit{class}. 
Two lattices are equivalent iff they belong to the same class. 
We denote by $[L]$ the class of the lattice $L$ and by $X^{0}$ the set of lattice classes.
\end{defi}

\begin{defi}\label{edge}
Two lattice classes $x,  \ x^{\prime}$ are \textit{adjacent} if there is a representative lattice $L^{\prime}$ of $x^{\prime}$ and $L$ of $x$ such that 
\[ \varpi L \subsetneq L^{\prime} \subsetneq L. \]
Denote by $X^{1}$ the set consisting of all unordered pairs $\{x, x^{\prime}\}$, for $x, x^{\prime}$ such that $x$ is adjacent to $x^{\prime}$. 
\end{defi}

\begin{defi} Denote by $X _{\PG_{2}(F)} = (X^{0}, X^{1})$ the graph where $X^{0}$, defined in Definition \ref{vert} forms the set of vertices and $X^{1}$, as defined in Definition \ref{edge} forms the set of edges. 
\end {defi}

We refer the reader to \cite[pp 70]{S} for a proof of the fact that $X_{\scriptscriptstyle\PG_{2}(F)}$ is a tree.

\begin{rem}\label{type q}
Suppose $V = F^{2}$ and denote by $e_{1},e_{2}$ the canonical basis of $V$.
Let $x_{0}$ denote the lattice class of $L_{0}:= \mathfrak{o}e_{1} \bigoplus \mathfrak{o}e_{2}$. 
Denote by $x$, a vertex adjacent to $x_{0}$ and by $L$, a representative of $x$. Since
\[\varpi L_{0}\subsetneq L \subsetneq L_{0} \Leftrightarrow L/\varpi L_{0} \subsetneq L_{0}/\varpi L_{0} \simeq \mathbb{F}_{q}\oplus \mathbb{F}_{q} \simeq \mathbb{F}_{q}^{2},\]
\noindent the neighbours of $x_{0}$ correspond to the one-dimensional subspaces of $\mathbb{F}_{q}^{2}$ which we can identify with the projective line $\mathbb{P}^{1}(\mathbb{F}_{q})$. 
It follows that $x_{0}$ has $q+1$ neighbours. 
It is easy to check that these correspond to the subspaces generated by  $\left(\begin{array}{ccc} 1 \\ 0 \end{array}\right)$ and $\left(\begin{array}{ccc} a \\ 1 \end{array}\right)$ for $a \in \mathbb{F}_{q}$.

\vspace{0.2 cm}
Since any lattice can be given the basis $e_{1}, e_{2}$ (using a change of bases matrix), we obtain that every vertex in $X_{\PG_{2}(F)}$ has $q+1$ neighbours.
Hence, the tree $X_{\PG_{2}(F)}$ is a Bruhat-Tits tree. 
\end{rem}

\begin{defi}
The apartment $\{x_n\}_{n \in \mathbb{Z}}$ such that $x_n=[L_n]$  where $L_n:= \mathfrak{o}e_1 + \mathfrak{o}\varpi^ne_2$ is called the \textit{standard apartment} of the tree $X_{\scriptscriptstyle\PG_{2}(F)}$.
\end{defi}

\subsection{Action of the group $\PG_{2}(F)$ on the tree $X_{\scriptscriptstyle\PG_{2}(F)}$}

The group $GL(V)$ of general linear transformations of $V$ has a natural action on $X_{\scriptscriptstyle\PG_{2}(F)}$ since it acts on the lattices of $V$. 
Fixing a basis $(e_{1}, e_{2})$ of $V$, we can identify this group with $\G_{2}(F)$.
Then we have a right action of $\G_{2}(F)$ on $X_{\scriptscriptstyle\PG_{2}(F)}$ given by
\[ \left(\begin{array}{ccc} a & b \\ c & d \end{array}\right)e_{1} =  ae_{1} + ce_{2},\  \left(\begin{array}{ccc} a & b \\ c & d \end{array}\right)e_{2} = be_{1} + de_{2}.\]

\noindent The centre $Z$ of $\G_{2}(F)$ (Definition \ref{subgroups}) is isomorphic to $F^{*}$ by the map $a \mapsto \left(\begin{array}{ccc} a & 0 \\ 0 & a \end{array}\right)$. 
This implies that $Z$ acts trivially on the vertices and therefore the action of the group $\G_{2}(F)$ descends to an action of the group $\PG_{2}(F)$ on $X_{\PG_{2}(F)}$. 

\vspace{0.2 cm}
Now we state a well known theorem describing the action of $\PG_{2}(F)$ on $X_{\scriptscriptstyle\PG_{2}(F)}$.

\begin{defi}\label{quot subgps}
Consider the image of the subgroups defined in Definition \ref{subgroups} under the quotient map $\G_{2}(F) \longrightarrow \G_{2}(F)/Z$.
We denote these to be ${\overline{K}},{\overline{B}}, {\overline{I}}$ corresponding to $K, B, I$.
\end{defi}

\begin{thm} \label{thm 2}
Denote by  $x_{i}$ the class of the lattice $L_{i}$ generated by $(e_{1},  \varpi^{i}e_{2})$, where $e_{1}$ and $e_{2}$ are the canonical basis of $V$. 
Let ${\overline{K}},{\overline{B}}, {\overline{I}}$ be as in Definition \ref{quot subgps}.
Then we have the following:

\begin{enumerate}
\item The group $\PG_{2}(F)$ acts transitively on the vertices and the edges of the tree $X_{\scriptscriptstyle\PG_{2}(F)}$.
\item The subgroup   ${\overline{K}}$ acts transitively on the spheres centred at $x_{0}$.
\item The stabiliser of the vertex $x_{0}$ is the subgroup  ${\overline{K}}$.
\item The stabiliser of the edge $\{x_{0}, x_{1}\}$ is the Iwahori subgroup $\overline{I}$.    
\item The stabiliser of the end $\omega = (x_{0},  x_{1},  x_{2},  \dots )$ is the Borel subgroup $\overline{B}$. 
\end{enumerate}
\end{thm}

\begin{proof}
Using matrix manipulation and Remark \ref{type q}. See \cite[pp 19]{K}.   
\end{proof}

\section{The automorphism group of a Bruhat-Tits tree}

In this section we discuss the automorphism group of a Bruhat-Tits tree and the property of weak two transitivity. In particular, we show that the subgroup $\PG_{2}(F)$ is a weakly two-transitive subgroup of the automorphism group of $X_{\scriptscriptstyle\PG_{2}(F)}$. 

\subsection {The group $\A(X)$}

Let $X = (X^{0},X^{1})$ be \textit{any} Bruhat-Tits tree. We recall some basic facts about the group of isometries of the Bruhat-Tits tree $X$.
We refer the reader to \cite[section $1.3$]{Ch} for a detailed introduction to the topology of this group and its action on the tree $X$.
  
\begin{defi} An \textit{automorphism} $\phi$ of the tree $X = (X^{0}, X^{1})$ is a bijection on the set of vertices $X^{0}$ which preserves the distance between them i.e., 
\[d(x, y) = d(\phi(x), \phi(y)), \ \forall \ x, y \in X^{0}. \]
\noindent The set of automorphisms of the tree $X$ will be denoted by $\A(X)$.
\end{defi}

\textit{The topology of $\A(X)$}:
Since an automorphism of the tree preserves the distance between any two vertices $x, y \in X^{0}$,  it is also a bijection on the set of edges $X^{1}$. 
The set $\A(X)$ of automorphisms of the tree is a group under the operation of composition. 
When viewed as the set of functions from $X^{0}$ to itself, it is a topological group equipped with the topology of pointwise convergence.   

\vspace{0.2 cm}
The following proposition tells us when is a subgroup of $\A(X)$ closed.

\begin{prop}[{\cite[pp $32$]{Ch}}]  \label{closed in A} For $x \in X^{0}$, we denote by $\mathcal{G}_{x}$ the stabiliser of the vertex $x$. 
A subgroup $H$ of $\A(X)$ is closed if and only if $H\cap \mathcal{G}_{x}$ is compact for a vertex $x$ of the tree $X$.  
\end{prop}

\vspace{0.2 cm}
We state now a basic definition which will be used in the proof of the main theorem.

\begin{defi}
An element $g$ of $\A(X)$ is said to \textit{invert} an edge $\alpha$ with extremities $x$ and $y$ if $g(x)= y$ and $g(y) = x$.     
\end{defi}

\begin{defi}
Let $g$ be an element of $\A(X)$ which does not invert any edges and $l_{g} := \mbox{min}\{d(x, g(x)): x \in X^{0}\}$. If $l_{g} = 0$ then $g(x) = x$ for some $x \in X^{0}$ and we call $g$ an \textit{elliptic} automorphism.
If $l_{g} > 0$,  we call $g$ a \textit{hyperbolic} automorphism.  
\end{defi}

\begin{rem}\label{auto types} 
Every automorphism of the tree which is not an inversion is either elliptic or hyperbolic. See \cite[pp 63,64]{S} for a proof. 
\end{rem}

\subsection{Weak two transitivity}

We discuss now the property of weak two transitivity and show in Proposition \ref{weak 2 transitivity} that a closed transitive subroup of $\A(X)$ will have this property. 
This result allows us to 'roughly' translate the transitive action of a group on the set of vertices to that on the set of ends.     

\begin{defi}\label{def wk k} A subgroup $\mathcal{G} \subset \A(X)$ is called \textit{weakly $k$-transitive} if for any two finite isomorphic subsets $A, B \subset X^{0}$ of cardinality $k$, parametrised by $x_{i} \in A$ and $y_{i} \in B$, there exists $g \in \mathcal{G}$ such that $g(x_{i}) = y_{i}$, for all $1 \leq i \leq k$.
\end{defi}

\vspace{0.2 cm}
The following proposition is similar to Proposition $1.4.1$ of \cite{Ch} and plays a crucial role in the proof of Theorem \ref{main thm}. Our assumptions are different: our tree is homogeneous and we have no numbering on the set of vertices. 
Most importantly, we have no condition on the distance between vertices which is used in Choucroun's proof. 
Instead, we use the fact that in our case, the action of the subgroup is transitive on the set of vertices. 
  
\begin{prop}\label{weak 2 transitivity} The following conditions are equivalent for any closed subgroup $\mathcal{G}$ of $\A(X)$ with transitive action on the set of vertices, $X^{0}$:
 
\begin{enumerate}
 
\item Given any two pairs of vertices $(x_{1}, x_{2}), (y_{1}, y_{2}$) such that $d(x_{1}, x_{2}) = d(y_{1}, y_{2})$, there exists an element $g \in \mathcal{G}$ such that $g(x_{i}) = y_{i}$ for $i = 1, 2$. 

\item The subgroup $\mathcal{G}$ is transitive on each sphere i.e., given any three vertices $x, \ y, \ z$ such that $d(x, y) = d(x, z)$, there exists $g \in \mathcal{G}$ such that $g(x)= x$ and $g(y) = z$.   

\item Let $\omega_{i}$ and $\sigma_{i}$ be ends and $[\omega_{1}, \omega_{2}]$,  $[\sigma_{1}, \sigma_{2}]$ apartments containing these ends. 
Consider $x_{0}$ and $y_{0}$ in the apartments $[\omega_{1}, \omega_{2}]$ and $[\sigma_{1}, \sigma_{2}]$, respectively. 
Then there exists $g \in \mathcal{G}$  such that $g(x_{0}) = y_{0}$ and $g(\omega_{i}) = \sigma_{i}$ for $i= 1, 2$.

\end{enumerate}

\end{prop} 

\begin{proof}

Clearly $(1)$ implies $(2)$. 
To prove the converse we consider two pairs of vertices $(x_{1}, x_{2})$  and $(y_{1}, y_{2})$. 
Since the subgroup $\mathcal{G}$ has transitive action on $X^{0}$, we know there exists an element $g \in\mathcal{G}$ such that $g(x_{1}) = y_{1}$.  
Let $g(x_{2}) = y_{2}^{\prime}$, for some vertex $y_{2}^{\prime}$. Then, $d(y_{1}, y_{2}^{\prime}) = d(x_{1}, x_{2}) = d(y_{1}, y_{2})$, since $g$ is an isometry. 
Therefore, $y_{2}$ and $y_{2}^{\prime}$ are in the same sphere centered at $y_{1}$. 
Now using condition $(2)$,  we know there exists $g^{\prime} \in \mathcal{G}$ such that $g^{\prime}(y_{1}) = y_{1}$ and $g^{\prime}(y^{\prime}_{2}) = y_{2}$. 
Hence, the element $g^{\prime}g \in \mathcal{G}$ and $g^{\prime}g(x_{i}) = y_{i}$ for $i= 1, 2$.

\vspace{0.2 cm}
The argument for $(1)$ implies $(3)$ and $(3)$ implies $(2)$ is the same as in \cite{Ch}. 
\end{proof} 

\begin{rem}  
Since Proposition \ref{weak 2 transitivity} ($1$) corresponds to the definition of weakly two-transitive, it implies that a closed subgroup of $\A(X)$ satisfying any of the equivalent conditions above is weakly two-transitive. 
\end{rem}

\section{The group $\mbox{\^{G}}^{(e)}$ and its decompositions}
In this final section we define the group $\mbox{\^{G}}^{(e)}$ and prove that it is a closed subgroup of $\A(X_{\scriptscriptstyle\PG_{2}(F)})$. 
We also show that it is weakly two-transitive (see Definition \ref{def wk k}). As detailed in the introduction, this group is of representation theoretic interest.

\subsection{The group $\mbox{\^{G}}^{(e)}$}
 
\begin{defi} Let $X_{\scriptscriptstyle\PG_{2}(F)} = \{X^{0}, X^{1}\}$ be the Bruhat-Tits tree associated to $\PG_{2}(F)$ and let $e$ be a positive integer.
For a vertex $x \in X^{0}$, we define \[B(x,e) := \{ y \in X^{0}: d(x, y) \leq e \}.\] 
Similarly for an edge $\eta \in X^{1}$, we define  \[ B(\eta,e) := \{ y \in X^{0}: d(\eta, y) \leq e \}\]

\noindent where $d(\eta, y) = \mbox{min}\{d(y, x_{1}),  d(y, x_{2})\}$ for $\eta = \{x_{1}, x_{2}\}$.
\end{defi}

\begin{defi}\label{G hat}
Let $e \geq 1$,  $\eta \in X^{1}$ and $B(\eta,e)$ as defined above. 
We denote by $\mbox{\^{G}}^{(e)}$ the set of automorphisms $g$ of $X_{\PG_{2}(F)}$ with the property that for all $\eta$ in $X^{1}$,  there is a $g^{\prime} \in \PG_{2}(F)$ such that 
the restrictions of $g$ and $g^{\prime}$ to $B(\eta,e)$ coincide,  i.e., {\small{
\[\mbox{\^{G}}^{(e)} := \{g \in \A(X_{\scriptscriptstyle\PG_{2}(F)}): \forall  \ \eta \in X^{1}, \  \exists  \ g^{\prime} \in \PG_{2}(F) \  \mbox{such that} \  g\vert_ {B(\eta,e)} = g^{\prime}\vert_{B(\eta,e)} \}.\]}}
\end{defi}

\vspace{0.2 cm}
The group $\mbox{\^{G}}^{(e)}$ is by definition a subgroup of $\A(X_{\scriptscriptstyle\PG_{2}(F)})$ containing $\PG_{2}(F)$. 
We have the following chain of group inclusions 
\[ \PG_{2}(F) \subset {\dots} \subset \mbox{\^{G}}^{(e+1)} \subset \mbox{\^{G}}^{(e)}\subset {\dots} \subset \mbox{\^{G}}^{(1)} \subset \A(X_{\scriptscriptstyle\PG_{2}(F)}).\]

\begin{rem}
It is easy to see that
\begin{enumerate}
\item If we let $e = 0$ in Definition \ref{G hat}, we obtain the group $\A(X_{\scriptscriptstyle\PG_{2}(F)})$.
\item  The intersection of $\mbox{\^{G}}^{(e)}$ for all $e \in \mathbb{N}$ is the subgroup $\PG_{2}(F)$.
\end{enumerate}

\end{rem}

\vspace{0.2 cm}
The following propositon is a well known fact.

\begin{prop}\label{closure}
Let $X_{\scriptscriptstyle\PG_{2}(F)}$ be the Bruhat-Tits tree associated to $\PG_{2}(F)$. Then the group $\PG_{2}(F)$ is a closed subgroup of $\A(X_{\scriptscriptstyle\PG_{2}(F)})$.
\end{prop}

\begin{proof}
Denote by $x_{0} = [L_{0}]$, where $L_{0} = e_{1}\mathfrak{o} \bigoplus e_{2}\mathfrak{o}$.
By Theorem \ref{thm 2}($3$) the stabiliser of $x_{0}$ in $\PG_{2}(F)$ is $\PG_{2}(\mathfrak{o})$. 
As noted in Remark \ref{K cpt}, this is a compact subgroup of $\PG_{2}(F)$. 
Since the topology of pointwise convergence coincides with the analytic topology on $\PG_{2}(F)$, $\PG_{2}({\mathfrak{o}})$ is a compact subgroup of $\A(X_{\scriptscriptstyle\PG_{2}(F)})$.  
Hence by Proposition \ref{closed in A}, $\PG_{2}(F)$ is a closed subgroup of $\A(X_{\scriptscriptstyle\PG_{2}(F)})$.  
\end{proof}

\vspace{0.2 cm}
Using this we prove the following Propositions. 

\begin{prop}\label{G hat closed}
For any $e\geq 1$, the group $\mbox{\^{G}}^{(e)}$ is a closed subgroup of $\A(X_{\scriptscriptstyle\PG_{2}(F)})$. 
\end{prop}

\begin{proof}

Let $(g_{n})_{n \in \mathbb{N}}$ be a sequence of functions in $\mbox{\^{G}}^{(e)}$ converging to $g$ under the topology of pointwise convergence i.e., 
$\displaystyle\lim_{n\rightarrow\infty}g_n(y) = g(y)$, for all $y$ in $X^{0}$.
To prove that $\mbox{\^{G}}^{(e)}$ is a closed subgroup, we need to show that $g \in \mbox{\^{G}}^{(e)}$. But, $g|_{B(x,e)}=\displaystyle\lim_{n\rightarrow\infty}g_n|_{B(x,e)}$. 
Since $\PG_2(F)$ is closed by Proposition \ref{closure}, we have $g|_{B(x,e)} \in \PG_2(F)$ (as $g_n \in \hat{G}^{(e)}$). Hence, $g \in \hat{G}^{(e)}$.


\end{proof}

\begin{prop}\label{PG weak 2 trans}
The group $\PG_{2}(F)$ is a weakly two-transitive subgroup of $\A(X_{\scriptscriptstyle\PG_{2}(F)})$.
\end{prop}

\begin{proof}
We know by Proposition \ref{closure} that $\PG_2(F)$ is closed in $\A(X_{\scriptscriptstyle{\PG_{2}(F)}})$. Since $\PG_2(\mathfrak{o}) \subset \PG_2(F)$, we have that $\PG_{2}(F)$ acts transitively on spheres (by Theorem \ref{thm 2}($2$)) in the sense
of Proposition \ref{weak 2 transitivity}($2$). 
Therefore, it is a weakly two-transitive subgroup of $\A(X_{\scriptscriptstyle\PG_{2}(F)})$. 
\end{proof}

\vspace{0.2 cm}
As a corollary we obtain.
  
\begin{cor}\label{G hat transitivity}
The subgroup $\mbox{\^{G}}^{(e)} \subset \A(X_{\scriptscriptstyle\PG_{2}(F)})$ is weakly two-transitive.   
\end{cor}

\begin{proof}
For any $e \geq 1$, the subgroup $\mbox{\^{G}}^{(e)}$ contains $\PG_{2}(F)$ which by Proposition \ref{PG weak 2 trans}, is weakly two-transitive. 
Hence the subgroup $\mbox{\^{G}}^{(e)}$ is weakly two-transitive.
\end{proof}

\subsection{Decompositions of $\mbox{\^{G}}^{(e)}$}

Now we prove the decompositions for $\mbox{\^{G}}^{(e)}$ parallel to the classical decompositions of $\G_{2}(F)$ stated in Theorem \ref{thm 1}. 
In contrast to the proof of the decompositions for $\G_{2}(F)$ which requires explicit computations with matrices (see \cite{B-H}), the proof for $\mbox{\^{G}}^{(e)}$ is purely geometric. 
Some parts of the proof given are similar to that of Theorem $1.5.2$ in \cite {Ch}.
Since our assumptions are different (see Proposition \ref{weak 2 transitivity}), we include these parts for the sake of completion. 

\vspace{0.2 cm}
Let $e \geq 1$. Denote by $\mbox{\^{G}} = \mbox{\^{G}}^{(e)}$. 
Let $|({x}_{n})|_{n\in \mathbb{Z}} = [\omega^{\prime}, \omega]$ be a fixed apartment with end points $\omega ^{\prime} = |({x}_{-n})|_{n \in \mathbb{N}}$ and $\omega = |({x}_{n})|_{n \in \mathbb{N}}$. 
We denote by $\mbox{\^{B}}$, the stabiliser in $\mbox{\^{G}}$ of the end $\omega$ and by $\mbox{\^{B}}^{\prime}$, the stabiliser of the end $\omega^{\prime}$. 
Let $(x_{n})_{n \in \mathbb{N}}$ be an infinite path in the equivalence class of $\omega$ and $g \in \mbox{\^{B}}$.
Then for $n \gg 0$, $g(x_{n}) = x_{n+k}$, for a fixed $k \in \mathbb{Z}$. 
The integer $k$ does not depend on the chosen representative of the class $\omega$. 

\begin{defi}\label{main defi} Define,   
\begin{align*}
&\mbox{\^{N}}: = \{ b\in \mbox{\^{B}} : b(x_{i}) = x_{i}, \mbox{ for some } {i} \in \mathbb{Z} \},\\
&\mbox{\^{N}}_{k}:= \{n \in \mbox{\^{N}} : n\vert_{[x_{k}, \omega]} = Id \},\\
&\mbox{\^{H}}:= \{ n \in \mbox{\^{N}}: n(x_{k}) = x_{k},  \ \forall \ k\in \mathbb{Z}\} = \displaystyle\bigcap_{k \in \mathbb{Z}}\mbox{\^{N}}_{k},\\
&\mbox{\^{K}}_{x_{0}}:= \{g \in  \mbox{\^{G}} : g(x_{0}) = x_{0}\},
\end{align*}

and the Iwahori subgroup \[\hspace{-18mm}\mbox{\^{I}}: = \{g \in  \mbox{\^{G}} : g(x_{o}) = x_{0}, \\\ g(x_{1}) = x_{1}\}.\]    
\end{defi}

\begin{lem}\label{a, t exist}
There exist elements $\alpha$ and $\tau$ in $\mbox{\^{G}}$ such that \[\alpha(x_{n}) = x_{-n}, \ \ \ \tau(x_{n}) = x_{n+1}  \ \ \   \forall \  n \in \mathbb{Z}. \]
\end{lem}

\begin{proof}
From Corollary \ref{G hat transitivity}, we know that the group $\mbox{\^{G}}$ is weakly two-transitive. 
Since both the apartments $[\omega, \omega^{\prime}]$ and $[\omega^{\prime}, \omega]$ contain $x_{0}$, by weak two transitivity,  there exists $\alpha$ in $\mbox{\^{G}}$ such that $\alpha(x_{0}) = x_{0}$ and $\alpha(\omega) =  \omega^{\prime}$.
Therefore, $\alpha(x_{n}) = x_{-n}$ for all $n \in \mathbb{Z}$. 
From the weak two transitivity of $\mbox{\^{G}}$,  there also exists an element $\tau$ in $\mbox{\^{G}}$ such that $\tau(x_{0}) = x_{1}$ and $\tau(\omega) = \omega$ as both $x_{0}$ and $x_{1}$ are in $[\omega^{\prime}, \omega]$.
This ends the proof of the lemma.            
\end{proof}

\begin{lem}\label{x0 x1} Given two vertices $x_{0}, x_{1}$ and an end $\sigma$, we have two alternatives, either $x_{0} \in [x_{1}, \sigma]$ or $x_{1} \in [x_{0}, \sigma]$. 
\end{lem}

\begin{proof} Suppose $x_{0} \notin [x_{1}, \sigma]$. Let $v_0$ be in the intersection of $[x_0,\sigma]$ and $[x_1,\sigma]$. Denote by $d(x_i,\sigma):=d(x_i,v_0)$.
Then, \[d(x_{1}, \sigma) < d(x_{0}, \sigma) + d(x_{0}, x_{1})\] \[\Rightarrow \ \ \ \ d(x_{1}, \sigma) + d(x_{0}, x_{1}) \leq d(x_{0}, \sigma) + 1. \ \ \ \ \ \ \ \  (\ast) \ \]
Equality in ($\ast$) implies that \[d(x_{1}, \sigma) = d(x_{0}, \sigma), \] which is absurd as $[x_{0}, x_{1}]$ is an edge.
Therefore,  we have \[d(x_{0}, \sigma) \overset {(1)}{\leq} d(x_{1}, \sigma) + d(x_{0}, x_{1})\leq d(x_{0}, \sigma).\] 
where $(1)$ is the triangle inequality. 
Hence, \[d(x_{1}, \sigma) + d(x_{0}, x_{1}) = d(x_{0}, \sigma). \] 
Therefore,  $x_{1} \in [x_{0}, \sigma]$.
We can use the same argument to show that if $x_{1} \notin [x_{0}, \sigma]$,  then $x_{0} \in [x_{1}, \sigma]$. 
\end{proof}

\begin{thm} \label{main thm} Using the notations of Definition \ref{main defi}, we have the following. 
\begin {enumerate}
 
\item The Iwasawa decomposition $\mbox{\^{G}} = \mbox{\^{K}}_{x_{0}}\mbox{\^{B}}$. 
\item The Cartan decomposition $\mbox{\^{G}} = \displaystyle\bigsqcup_{n \in \mathbb{N}}\mbox{\^{K}}_{x_{0}}{\tau}^{n}\mbox{\^{K}}_{x_{0}}$. 
\item The Bruhat decomposition $\mbox{\^{G}} = \mbox{\^{N}}\alpha \mbox{\^{B}}\sqcup \mbox{\^{B}}$. 
Also,  \[n_{1}\alpha \mbox{\^{B}} \cap n_{2}\alpha \mbox{\^{B}} \neq \varnothing \Leftrightarrow n_{1}\mbox{\^{H}} = n_{2}\mbox{\^{H}}.\]

\item The Levi decomposition $\mbox{\^{B}} = \mbox{\^{T}}\mbox{\^{N}}$, where $\mbox{\^{T}} = \{\tau^{n}: n \in \mathbb{Z} \}$.
\item The Iwahori group admits a decomposition $\mbox{\^{I}} = (\mbox{\^{I}} \cap \mbox{\^{B}})(\mbox{\^{I}} \cap \mbox{\^{B}}^{\prime})$ and we have the intersection $(\mbox{\^{I}} \cap \mbox{\^{B}}) \cap (\mbox{\^{I}} \cap \mbox{\^{B}}^{\prime}) = \mbox{\^{H}}$.
\item The decomposition $\mbox{\^{K}}_{x_{0}} = \mbox{\^{I}}\sqcup \mbox{\^{I}} \alpha \mbox{\^{I}}$ and the decomposition $\mbox{\^{G}} = \mbox{\^{I}}\mbox{\^{B}}\sqcup \mbox{\^{I}}\alpha \mbox{\^{B}}$.
\item The indices $[\mbox{\^{N}}_{k}: \mbox{\^{N}}_{k-1}] = q$  for all $k \in \mathbb{Z}$.   

\end{enumerate}
\end {thm}

\begin{proof}
We begin by showing that $\mbox{\^{K}}_{x_{0}}$ acts transitively on the set of ends $\Omega$ of the tree. 
Consider the vertex $x_{0}$ and ends $\omega$ and $\sigma$. By Remark \ref{apart}, there exist infinite paths $[x_{0}, \omega]$ and $[x_{0}, \sigma]$. 
From the weak two transitivity of $\mbox{\^{G}}$, there exists an element $g \in  \mbox{\^{G}}$ such that $g(x_{0}) = x_{0}$ and $g(\omega) = \sigma$. 
Since $x_{0}$ is fixed,  such a $g$ is necessarily in $\mbox{\^{K}}_{x_{0}}$.
Hence, $\mbox{\^{K}}_{x_{0}}$ acts transitively on the boundary. 

\vspace{0.2 cm}
Now consider an element $g \in \mbox{\^{G}}$ and the end $g(\omega)$. 
Since $\mbox{\^{K}}_{x_{0}}$ acts transitively on $\omega$,  we have for any $g$,  $g(\omega) = k(\omega)$ for some $k$ in $\mbox{\^{K}}_{x_{0}}$. 
Hence,  $k^{-1}g$ fixes $\omega$,  therefore there exists $b \in \mbox{\^{B}}$ such that $k^{-1}g = b$ which implies $g= kb$.
Since $g$ was arbitrary,  we have the Iwasawa decomposition $\mbox{\^{G}} = \mbox{\^{K}}_{x_{0}}\mbox{\^{B}}$.

\vspace{0.2 cm}
Let $g\in \mbox{\^{G}}$ be arbitrary. 
If $g$ fixes $x_{0}$,  then $g$ belongs to $\mbox{\^{K}}_{x_{0}}$ and the result is clear. 
If not,  let $d(x_{0}, g(x_{0}))= n $. 
From the transitive action of $\mbox{\^{K}}_{x_{0}}$ on $\Omega$ (shown above), it follows that it acts transitively on each sphere around ${x_{0}}$ (Proposition \ref{PG weak 2 trans}). 
Therefore,  there exists $k \in \mbox{\^{K}}_{x_{0}}$ such that $k(g(x_{0})) = x_{n}$. 
Applying $\tau^{-n}$ gives us an element $\tau^{-n}kg \in \mbox{\^{K}}_{x_{0}}$, since $\tau^{-n}kg$ fixes $x_{0}$. 
Hence,  $g$ is an element of $\mbox{\^{K}}_{x_{0}}\tau^{n}\mbox{\^{K}}_{x_{0}}$. 
Note that the cosets $\mbox{\^{K}}_{x_{0}}\tau^{n}\mbox{\^{K}}_{x_{0}}$ are disjoint for different $n$.
This gives us the Cartan decomposition for $\mbox{\^{G}}$.

\vspace{0.2 cm}
To show the Bruhat decomposition for $\mbox{\^{G}}$, it suffices to prove that if $\sigma$ is an end different from $\omega$,  there exists $n \in \mbox{\^{N}}$ such that $n(\alpha(\omega)) = n(\omega^{\prime}) = \sigma$.
If $\sigma = \omega^{\prime}$,  then we can take $n = \mathrm{Id} \in \mbox{\^{N}}$. 
Otherwise,  we consider the crossroad $x_{k}$ of the three ends $\omega, \omega^{\prime}, \sigma$. 
We know from Proposition \ref{crossroad 2} that such an element exists.
Using the weak two transitivity of $\mbox{\^{G}}$  we have an element $g \in\mbox{\^{G}}$ such that $g[\omega^{\prime},  \omega] = [\sigma, \omega] $ and $g(x_{k}) = x_{k}$.  
As $g$ fixes $\omega$ and $x_{k}$,  it follows $g\in \mbox{\^{N}}$. 
Therefore,  we have the Bruhat decomposition $\mbox{\^{G}} = \mbox{\^{N}} \alpha \mbox{\^{B}}\sqcup \mbox{\^{B}}$. 
Moreover, 
 \[n_{1}\alpha \mbox{\^{B}} \cap n_{2} \alpha \mbox{\^{B}} \neq \varnothing \Leftrightarrow n_{1}(\omega^{\prime}) = n_{2}(\omega^{\prime}) \Leftrightarrow n_{1}^{-1}n_{2}(\omega^{\prime}) = \omega^{\prime}.\] 
This implies that $ n = n_{1}^{-1}n_{2}$ stabilises the apartment $[\omega^{\prime}, \omega]$ since $n^{-1}_{1}n_{2} \in \mbox{\^{B}}^{\prime} \cap \mbox{\^{B}}$. 
As $n \in \mbox{\^{N}}$ it also fixes this apartment.
Since $\mbox{\^{H}}$ is the maximal subgroup fixing the apartment $[\omega^{\prime},  \omega]$,  it implies that $n_{1}^{-1}n_{2}\in \mbox{\^{H}}$. 
Hence, we have the equality $n_{1}\mbox{\^{H}} = n_{2}\mbox{\^{H}}$. 

\vspace{0.2 cm}
Let $g \in \mbox{\^{B}}$. Then $g$ cannot act by inversion since it stabilises an end. 
Hence,  $g$ is either elliptic or hyperbolic (Remark \ref{auto types}). 
If it is elliptic then it is in the subgroup $\mbox{\^{N}}$ by definition.
If not,  it is hyperbolic and therefore there exists $m_{0} \in \mathbb{N}$ and $n \in \mathbb{Z}$ such that 
      \[ g(x_{m}) = x_{m+n},  \ \ \ \ \forall \ m > m_{0}. \] 
 This implies that $\tau^{-n}g$ is elliptic and therefore an element of $\mbox{\^{N}}$ i.e.,  $\tau^{-n}g = h$ for some $h \in \mbox{\^{N}}$.
Therefore $g = {\tau}^{n}h$. This proves the Levi decomposition. 

\vspace{0.2 cm}
Let $g\in \mbox{\^{I}}$ and $g^{-1}(\omega) = \sigma$. 
By definition $g$ fixes $x_{0}$ and $x_{1}$.
We consider the image of $[x_{0}, \omega]$ under $g$. It is the infinite path $[x_{0}, \sigma]$ which contains $x_{1}$. 
Since $\mbox{\^{G}}$ is weakly two-transitive,  there exists an element $h \in \mbox{\^{G}}$ which maps $[\omega^{\prime}, \omega]$ to $[\omega^{\prime}, \sigma]$ and $x_{0}$ to itself. 
As $h$ fixes $x_{0}$ as well as $x_{1}$,  $h \in \mbox{\^{I}}$.  
By construction $h(\omega^{\prime}) = \omega^{\prime}$,  therefore $h \in \mbox{\^{I}} \cap \mbox{\^{B}}^{\prime}$. 
This implies $h^{-1}\in \mbox{\^{I}}\cap\mbox{\^{B}}^{\prime}$,  as $\mbox{\^{I}}\cap\mbox{\^{B}}^{\prime}$ is a group. 
Now $gh(\omega) = g(\sigma) = \omega$.
Therefore $gh \in \mbox{\^{I}} \cap \mbox{\^{B}}$, so $ghh^{-1}  = g \in (\mbox{\^{I}}\cap \mbox{\^{B}})(\mbox{\^{I}} \cap B^{\prime})$.
Hence, the Iwahori subgroup  $\mbox{\^{I}}$ admits the decomposition $\mbox{\^{I}} = (\mbox{\^{I}}\cap \mbox{\^{B}})(\mbox{\^{I}} \cap B^{\prime})$.

\vspace{0.2 cm}
Since the group $\mbox{\^{B}} \cap \mbox{\^{B}}^{\prime}$ stabilises the ends of the geodesic $[\omega^{\prime}, \omega]$,  its intersection with $\mbox{\^{I}}$ fixes a vertex. 
Therefore,  $\mbox{\^{B}} \cap \mbox{\^{B}}^{\prime} \cap \mbox{\^{I}}$ fixes the apartment $[\omega^{\prime}, \omega]$ and is contained in the group $\mbox{\^{H}}$. 
Since $\mbox{\^{H}}$ is contained in $\mbox{\^{B}}, \mbox{\^{B}}^{\prime}, \mbox{\^{I}}$,  we conclude $(\mbox{\^{I}} \cap \mbox{\^{B}}) \cap (\mbox{\^{I}} \cap \mbox{\^{B}}^{\prime}) = \mbox{\^{H}}$.

\vspace{0.2 cm}
In order to prove that $\mbox{\^{K}}_{x_{0}}$ has the decomposition $\mbox{\^{I}}\sqcup \mbox{\^{I}} \alpha \mbox{\^{I}}$,  we consider an element $k \in \mbox{\^{K}}_{x_{0}}$ and the vertex $k(x_{1})$. 
If $k(x_{1}) = x_{1}$,  then $k$ is in the Iwahori subgroup $\mbox{\^{I}}$. 
Otherwise,  we can apply  weak two transitivity of $\mbox{\^{G}}$ to the paths $[x_{-1}, x_{1}]$ and $[k(x_{1}), x_{1}]$, both of which contain $x_{0}$. 
Hence,  there exists an element $i \in \mbox{\^{I}}$  such that $i(x_{-1}) = k(x_{1})$ and $i(x_{1}) = x_{1}$. 
Since $i\alpha$ and $k$ have the same image on $x_{0}$ and $x_{1}$,  we have $k \in i\alpha \mbox{\^{I}}$.
Hence, $\mbox{\^{K}}_{x_{0}} = \mbox{\^{I}}\sqcup\mbox{\^{I}}\alpha\mbox{\^{I}}$.

\vspace{0.2 cm}
We now prove the decomposition $\mbox{\^{G}} = \mbox{\^{I}}\mbox{\^{B}}\sqcup\mbox{\^{I}}\alpha\mbox{\^{B}}$. 
Let $g \in \mbox{\^{G}}$ and $\sigma = g(\omega)$
Using Lemma \ref{x0 x1}, we have two cases. 

\vspace{0.2 cm}
\begin{case} Assume $x_{1} \in [x_{0}, \sigma]$. Since $\mbox{\^{K}}_{x_{0}}$ has transitive action on the set of ends $\Omega$, there exists  an element $h \in \mbox{\^{K}}_{x_{0}}$ such that $h([x_{0},\omega]) = [x_{0},\sigma]$. 
In particular, $h(x_{1}) = x_{1} $ and $h(\omega) = \sigma$. 
Therefore, $h \in \mbox{\^{I}}$. 
Since the image of $\omega$ is the same under $g$ and $h$ differing only by an element of $\mbox{\^{B}}$,  it follows that $h^{-1}g \in \mbox{\^{B}}$. Thus,  $g \in \mbox{\^{I}}\mbox{\^{B}}$.
\end{case}

\vspace{0.2 cm}
\begin{case} Assume $x_{0} \in [x_{1}, \sigma]$. The subgroup in $\mbox{\^{G}}$ stabilizing $x_{1}$,  denoted $\mbox{\^{K}}_{x_{1}}$ has a transitive action on $\Omega$ for the same reason as the subgroup $\mbox{\^{K}}_{x_{0}}$. 
Hence, there exists an element $p \in \mbox{\^{G}}$ such that $p(x_{1}) = x_{1}$ and $p(\omega^{\prime}) = \sigma$. 
Since $x_{0}$ is contained in both the infinite paths $[x_{1}, \omega^{\prime}]$ and $p([x_{1}, \omega^{\prime}]) = [x_{1}, \sigma]$,  we can deduce that $p \in \mbox{\^{I}}$. 
As the image of $\omega$ is the same under $g$ and $p\alpha$, it follows that $g\in \mbox{\^{I}}\alpha B$. 

\vspace{0.2 cm}
\noindent Therefore, $\mbox{\^{G}} = \mbox{\^{I}}\mbox{\^{B}}\sqcup \mbox{\^{I}}\alpha \mbox{\^{B}}$. 
\end{case}

\vspace{0.2 cm}
 Consider the map \begin {align*}
                      \phi: \mbox{\^{N}}_{k}/\mbox{\^{N}}_{k-1} & \xrightarrow{\sim} S(x_{k}, 1) -\{x_{k+1}\} \\
                                                   g & \longmapsto     g(x_{k-1}).
\end{align*}
We claim $\phi$ is a bijection. Indeed, note that \[ g_{1}(x_{k-1}) = g_{2}(x_{k-1}) \Rightarrow  g_{1}^{-1}g_{2}(x_{k-1}) = x_{k-1}.\]
This implies $g_{1}^{-1}g_{2} \in \mbox{\^{N}}_{k-1}$.
Hence,  they define the same class on $\mbox{\^{N}}_{k}/\mbox{\^{N}}_{k-1}$ and therefore the map $\phi$ is injective.
So to prove that the index $[\mbox{\^{N}}_{k}:\mbox{\^{N}}_{k-1}] = q$, it suffices to show the transitivity of $\mbox{\^{N}}_{k-1}$ on $S(x_{k}, 1) -\{x_{k+1}\}$, which is a set with cardinality $q$. 

\vspace{0.2 cm}
Consider a vertex $u \in S(x_{k}, 1) - \{x_{k+1}\}$ and the infinite paths $[x_{k-1},  \omega]$ and $[u, \omega]$. 
From the weak two transitivity of $\mbox{\^{G}}$,  there exists $g \in \mbox{\^{G}}$ such that $g(x_{k-1}) = u$ and $g(\omega) =\omega$. 
Since $g(x_{k}) = x_{k}$,  it implies $g \in \mbox{\^{N}}_{k}$.
This proves the transitivity of $\mbox{\^{N}}_{k}$ on $S(x_{k}, 1)- \{x_{k+1}\}$.
Hence,  $[\mbox{\^{N}}_{k}:\mbox{\^{N}}_{k-1}] = q$.
 
\end{proof}

\section{Closing Remarks}
\begin{enumerate}

\item Theorem \ref{main thm} applies to any closed transitive subgroup of the automorphism group of any Bruhat-Tits tree. 
In particular, we get the classical decomposition of $\G_{2}(F)$ by a geometric proof.

\item As in the case of $\G_{2}(F)$, one can define for a vertex $x_{0}$, the $mth$ congruence subgroups in $\mbox{\^{G}}$, i.e. those elements which stabilize a ball of radius $m$ centred at $x_{0}$.
 This has been done in \cite{K}, where we also study certain properties of these subgroups..

\item Theorem \ref{main thm} and the previous remark suggest that one could study the representation theory of the group $\mbox{\^{G}}$ over $\overline{\mathbb{F}}_{p}$-vector spaces in an analogous manner to that of $\G_{2}(F)$ as covered in \cite{B-L1} and \cite{B-L2}.    
 
\end{enumerate}

\bibliographystyle{alpha}
\bibliography{masterbib}

\begin{thebibliography}{{Ol'}76}

\bibitem[BH06]{B-H}
C.~J. Bushnell and G.~Henniart.
\newblock {\em The local {L}anglands conjecture for $\G_{2}(F)$}, volume 335 of
  {\em Grundlehren der Mathematischen Wissenschaften}.
\newblock Springer-Verlag, 2006.

\bibitem[BK96]{B-K}
T.N. Bailey and A.W. Knapp.
\newblock {\em Representation theory and automorphic forms}.
\newblock Instructional conference, International Centre for Mathematical
  Sciences, Edinburgh, Scotland. 1996.

\bibitem[BL94]{B-L1}
L.~Barthel and R.~Livne.
\newblock Irreducible modular representations of {$\mathrm{GL}_{2}(F)$} of a
  local field.
\newblock {\em Duke Math. J.}, 75(2):261--292, 1994.

\bibitem[BL95]{B-L2}
L.~Barthel and R.~Livne.
\newblock Modular representations of {$\mathrm{GL}_{2}(F)$} of a local field:
  The ordinary, unramified case.
\newblock {\em Journal of Number Theory}, 55:1--27, 1995.

\bibitem[Cho94]{Ch}
F.M Choucroun.
\newblock {\em Analyse harmonique des groupes d'automorphisms d'arbres de
  {B}ruhat-{T}its}, volume~58.
\newblock Mémoires de la Société Mathématique de France, 1994.

\bibitem[DT07]{DT}
S.~Dasgupta and J.~Teitelbaum.
\newblock {\em The p-adic upper half plane}.
\newblock Lectures for the $(2007)$ Arizona Winter School, 2007.

\bibitem[Kau12]{K}
I.~Kaur.
\newblock Decomposition theorems for groups acting on the {B}ruhat-{T}its tree.
\newblock Master's thesis, Humboldt Universit{\"{a}}t zu Berlin, 2012.

\bibitem[{Ol'}76]{O}
G.I. {Ol'shanskii}.
\newblock Classification of irrreducible representations of groups of
  automorphisms of the {B}ruhat-{T}its trees.
\newblock {\em Funkts. Anal.i. Prilozh}, 11:32--42, 1976.

\bibitem[Ser03]{S}
J.~P. Serre.
\newblock {\em Trees}.
\newblock Springer Monographs in Mathematics. Springer-Verlag, 2003.

\end{thebibliography}

\vspace{0.4 cm}
\noindent{\small{Freie Universität Berlin, FB Mathematik und Informatik, Arnimallee 3, 14195 Berlin, Germany.
\\ \textit{E-mail address}: kaur@math.fu-berlin.de}}
\end{document}